\documentclass[12pt, reqno]{article}
\usepackage{amsmath}
\usepackage{amssymb}
\usepackage{latexsym,bm}
\usepackage{amsthm}
\usepackage{graphicx}
\usepackage{cite}
\usepackage[colorlinks,urlcolor=blue]{hyperref}

\usepackage[top=3cm,bottom=3cm,left=3cm,right=3cm,includehead,includefoot]{geometry}
\usepackage{algorithm}
\usepackage{algorithmic}
\allowdisplaybreaks

\newtheorem{thm}{Theorem}[section]

\newtheorem{lem}{Lemma}[section]

\newtheorem{ass}{Assumption}[section]
\theoremstyle{definition}
\newtheorem{defn}{Definition}[section]
\theoremstyle{Condition}

\theoremstyle{remark}

\numberwithin{equation}{section}
\theoremstyle{example}

\numberwithin{equation}{section}

\begin{document}

\bigskip
\bigskip

\bigskip

\begin{center}

\textbf{\large A stochastic coordinate descent splitting primal-dual
fixed point algorithm and applications to large-scale composite
optimization}

\end{center}

\begin{center}
Meng Wen $^{1}$, Yu-Chao Tang$^{2}$, Jigen Peng$^{1}$
\end{center}

\begin{center}
1. School of Mathematics and Statistics, Xi'an Jiaotong University,
Xi'an 710049, P.R. China \\
2. Department of Mathematics, NanChang
University, Nanchang 330031, P.R. China
\end{center}

\footnotetext{\hspace{-6mm}$^*$ Corresponding author.\\
E-mail address: wen5495688@163.com}

\bigskip

\noindent  \textbf{Abstract} In this paper, we consider the problem
of finding the minimizations of the sum  of two convex functions and
the composition of another convex function with a continuous linear
operator from the view of fixed point algorithms based on proximity
operators, which is is inspired by recent results of Chen, Huang and
Zhang. With the idea of coordinate descent, we design a stochastic
coordinate descent splitting primal-dual fixed point algorithm.
Based on randomized krasnosel'skii mann iterations and the firmly
nonexpansive properties of the proximity operator, we achieve the
convergence of the proposed  algorithms. Moreover, we give two
applications of our method. (1) In the case of stochastic minibatch
optimization, the algorithm can be applicated to split a composite
objective function into blocks, each of these blocks being processed
sequentially by the computer. (2) In the case of distributed
optimization, we consider a set of $N$ networked agents endowed with
private cost functions and seeking to find a consensus on the
minimizer of the aggregate cost. In that case, we obtain a
distributed iterative algorithm where isolated components of the
network are activated in an uncoordinated fashion and passing in an
asynchronous manner.  Finally, we illustrate the efficiency of the
method in the framework of large scale machine learning
applications. Generally speaking, our method SCDSPDFP$^{2}O$ is
comparable with other state-of-the-art methods in numerical
performance, while it has some advantages on parameter selection in
real applications.


\noindent \textbf{MR(2000) Subject Classification} 47H09, 90C25,

\section{Introduction}

In this paper, we aim at solving the following minimization problem
$$\min_{x\in\mathcal{X}} f(x)+g(x)+ (h\circ D)(x),\eqno{(1.1)}$$
where $\mathcal{X}$ and $\mathcal{Y}$ are two Euclidean spaces,
$f,g\in\Gamma_{0}(\mathcal{X})$, $h\in\Gamma_{0}(\mathcal{Y}),$ and
$f$ is differentiable on $\mathcal{Y}$ with a $1/\beta$-Lipschitz
continuous gradient for some $\beta\in(0,+\infty)$ and
$D:\mathcal{X}\rightarrow\mathcal{Y}$ a linear transform. This
parameter $\beta$ is related to the convergence conditions of
algorithms presented in the following section. Here and in what
follows, for a real Hilbert space $\mathcal{H}$,
$\Gamma_{0}(\mathcal{H})$ denotes the collection of all proper lower
semi-continuous convex functions from $\mathcal{H}$ to
$(-\infty,+\infty]$. Despite its simplicity, when $g=0$ many
problems in image processing can be formulated in the form of (1.1).
For instance, the following variational sparse recovery models are
often considered in image restoration and medical image
reconstruction:
$$\min\frac{1}{2}\|Ax-b\|_{2}^{2}+\lambda\psi(Dx),\eqno{(1.2)}$$
where $\|\cdot\|_{2}$ denotes the usual Euclidean norm for a vector,
$A\in\mathbb{R}^{p\times n}$ describes  a blur operator, $b\in
\mathbb{R}^{p}$ represents the blurred and noisy image  and
$\lambda>0$ is the regularization parameter in the context of
deblurring and denoising of images.
\par
  For
problem (1.2),  Chen et al proposed
 a primal-dual fixed point algorithm($PDFP^{2}O)$ in [1], i.e.
$$
\left\{
\begin{array}{l}
v_{n+1}=(I-prox_{\frac{\gamma}{\lambda}f_{1}})(D(x_{n}-\gamma\nabla f_{2}(x_{n}))+(I-\lambda DD^{T})v_{n}),\\
x_{n+1}=x_{n}-\gamma\nabla f_{2}(x_{n})-\lambda D^{T}v_{n+1},
\end{array}
\right.\eqno{(1.3)}
$$
where $0<\lambda\leq1/\lambda_{\max}(DD^{T})$, $0<\gamma<2\beta$,
and the operator $ prox_{f}$ is called the proximity operator of $f$
. Note that this type of splitting method was originally studied in
[10,11] and the notion of proximity operators was first introduced
by Moreau in [12] as a generalization of projection operators.
Motivated and inspired by the above results, we introduced a
 splitting primal-dual fixed point
algorithm. The contributions of us are the following aspects:
\par
(I) The algorithm that we proposed includes the well known PFPS [13]
and $FP^{2}O$ [14] as a special case. Moreover, the idea based on
the results of  Chen et al [1], and the obvious advantage of the
proposed scheme is that it is very easy for parallel implementation.
\par
(II) Based on the results of Chen et al [1] and Bianchi et al [2],
we introduce the idea of stochastic coordinate descent on splitting
primal-dual fixed point algorithm. The form of splitting primal-dual
fixed point algorithm can be translated into fixed point iterations
of a given operator having a nonexpansive property. By the view of
stochastic coordinate descent, we know that at each iteration, the
algorithm is only to update  a random subset of coordinates.
Although this leads to a perturbed version of the initial splitting
primal-dual fixed point iterations, but it can be proved to preserve
the convergence properties of the initial unperturbed version.
Moreover, stochastic coordinate descent has been used in the
literature [15-17] for proximal gradient algorithms. We believe that
its application to splitting primal-dual fixed point algorithm well
suited to large-scale optimization problems.
\par
(III) We use our views to large-scale optimization problems which
arises in signal processing and machine learning contexts. We prove
that the general idea of stochastic coordinate descent gives a
unified framework allowing to derive stochastic algorithms of
different kinds. Furthermore, we give two application examples.
Firstly, we propose a new stochastic approximation algorithm by
applying stochastic coordinate descent on the top of SPDFP$^{2}$O.
The algorithm is called as stochastic minibatch splitting
primal-dual fixed point algorithm (SMSPDFP$^{2}$O) Secondly, we
introduce a random asynchronous distributed optimization methods
that we call as distributed asynchronous splitting primal-dual fixed
point algorithm (DASMSPDFP$^{2}$O). The algorithm can be used to
efficiently solve an optimization problem over a network of
communicating agents.  The algorithms are asynchronous in the sense
that some components of the network are allowed to wake up at random
and perform local updates, while the rest of the network stands
still. No coordinator or global clock is needed. The frequency of
activation of the various network components is likely to vary.

The rest of this paper is organized as follows. In the next section,
 we introduce some notations used throughout in the paper. In section 3, we devote
to introduce  SPDFP$^{2}$O algorithm and its relation with the
PDFP$^{2}$O, we also show how the SPDFP$^{2}$O includes PDFP$^{2}$O
as a special case. In section 4, we propose a stochastic
approximation algorithm from the SPDFP$^{2}$O.  In section 5, we
addresse the problem of asynchronous distributed optimization. In
the final section, we show the numerical performance and efficiency
of propose algorithm through some examples in the context of
large-scale $l_{1}$-regularized logistic regression.

\section{Preliminaries }
Throughout the paper, we  denote by $\langle \cdot, \cdot\rangle$
the inner product on $\mathcal{X}$ and by $\|\cdot\|$ the norm on
$\mathcal{X}$. We consider the case where $D$ is injective(in
particular, it is implicit that dim$(\mathcal{X})\leq$
dim$(\mathcal{Y}))$. In the latter case, we denote by $\mathcal{R}$
= Im$(D)$ the image of $D$ and by $D^{-1}$ the inverse of $D$ on
$\mathcal{R}\rightarrow\mathcal{X}$. We emphasize the fact that the
inclusion $\mathcal{R}\subset\mathcal{Y}$ might be strict. We denote
by $\nabla$ the gradient operator. We make the following
assumptions:
\begin{ass}
The following facts holds true:\\
 (1)$D$ is injective;\\
 (2)$f$  has $1/\beta$-Lipschitz continuous gradient.
\end{ass}
\begin{ass}
The infimum of problem (1.1) is attained. Moreover, the following
qualification condition holds
$$0\in ri(dom\, h-D\, dom\, g).$$
\end{ass}
\begin{defn}
 Let $f$ be  a real-valued convex function on
$\mathcal{X}$, the operator prox$_{f}$ is defined by
\begin{align*}
prox_{f}&:\mathcal{H}\rightarrow\mathcal{H}\\
& x\mapsto \arg \min_{y\in H}f(y)+\frac{1}{2}\|x-y\|_{2}^{2},
\end{align*}
called the proximity operator of $f$.

\end{defn}

\begin{defn}
Let $A$ be a closed convex set of $\mathcal{X}$. Then the indicator
function of $A$ is defined as
$$
\iota_{A}(x) = \left\{
\begin{array}{l}
0,\,\,\, \,\,if x\in A,\\
\infty,\,\,\, otherwise .
\end{array}
\right.
$$
\end{defn}
It can easy see the proximity operator of the indicator function in
a closed convex subset $A$ can be reduced a projection operator onto
this closed convex set $A$. That is,
$$prox_{\iota_{A}}=proj_{A}$$
where proj is the projection operator of $A$.

\begin{defn}
(Nonexpansive operators and firmly nonexpansive operators [4]). An
operator $T : \mathcal{H} \rightarrow {\mathcal{H}}$ is nonexpansive
if and only if it satisfies
$$\|Tx-Ty\|_{2}\leq\|x-y\|_{2}\,\,\, for\,\,all\,\,\, (x,y)\in \mathcal{H}^{2}.$$
$T$ is firmly nonexpansive if and only if it satisfies one of the
following equivalent conditions:
\par
(i)$\|Tx-Ty\|_{2}^{2}\leq\langle Tx-Ty,x-y\rangle$\,\,\,
for\,\,all\,\,\, $(x,y)\in \mathcal{H}^{2}$;
\par
(ii)$\|Tx-Ty\|_{2}^{2}=\|x-y\|_{2}^{2}-\|(I-T)x-(I-T)y\|_{2}^{2}$\,\,\,
for\,\,all\,\,\, $(x,y)\in \mathcal{H}^{2}$.
\par
It is easy to show from the above definitions that a firmly
nonexpansive operator $T$ is nonexpansive.
\end{defn}
\begin{lem}
(Lemma 2.4 of [3]). Let $f$ be a function in
$\Gamma_{0}(\mathcal{X})$. Then $prox_{f}$ and $I- prox_{f}$ are
both firmly nonexpansive operators.
\end{lem}

For an element $u=(v,x)\in \mathcal{Y}\times\mathcal{X}$, with $v\in
\mathcal{Y}$ and $x\in\mathcal{X}$, let
$$\|u\|_{\lambda}=\sqrt{\|x\|_{2}^{2}+\lambda\|v\|_{2}^{2}}.$$
 We can easily see that
$\|\cdot\|_{\lambda}$ is a norm over the produce space
$\mathcal{Y}\times\mathcal{X}$ whenever $\lambda > 0$.

\begin{lem}
([1]). Let Assumptions 2.2 hold true.  If $0<\gamma<2\beta$,
$0<\lambda\leq1/\lambda_{max}(\tilde{D}\tilde{D}^{T})$, Let
$(\tilde{v}^{k+1}, x^{k+1}) = T(\tilde{v}^{k}, x^{k})$ where $T$ is
the transformation described by Equations (3.3). Then $T$ is
nonexpansive under the norm $\|\cdot\|_{\lambda}$.
\end{lem}

\begin{defn}
(Randomized krasnosel'skii mann iterations[2]). Let $\mathcal{V}$ be
a Euclidean space. Consider the space
$\mathcal{V}=\mathcal{V}_{1}\times\cdots\times\mathcal{V}_{J}$ for
some $J\in\mathbb{N}^{\ast}$ where for any $j$, $\mathcal{V}_{j}$ is
a Euclidean space. For $\mathcal{V}$ equipped with the scalar
product $\langle x,y\rangle=\sum_{j=1}^{J}\langle
x_{j},y_{j}\rangle_{\mathcal{V}_{j}}$ where $\langle
\cdot,\cdot\rangle_{\mathcal{V}_{j}}$ is the scalar product in
$\mathcal{V}_{j}$. For $j\in \{1,\cdots,J\}$ , let $T_{j}:
\mathcal{V}\rightarrow\mathcal{V}_{j}$ be the components of the
output of operator $T : \mathcal{V}\rightarrow\mathcal{V}$
corresponding to $\mathcal{V}_{j}$ , so, we have
$Tx=(T_{1}x,\cdots,T_{J}x)$. Let $2^{\mathcal{J}}$ be the power set
of $\mathcal{J}=\{1,\cdots,J\}$. For any $\kappa\in
2^{\mathcal{J}}$, we donate the operator $\hat{T}^{\kappa}:
\mathcal{V}\rightarrow\mathcal{V}$ by $\hat{T}^{\kappa}_{j}x=T_{j}x$
for $j\in\kappa$ and $\hat{T}^{\kappa}_{j}x=x_{j}$ for otherwise. On
some probability space $(\Omega, \mathcal{F}, \mathbb{P})$, we
introduce a random i.i.d. sequence
$(\zeta^{k})_{k\in\mathbb{N}^{\ast}}$ such that
$\zeta^{k}:\Omega\rightarrow 2^{\mathcal{J}}$ i.e.
$\zeta^{k}(\omega)$ is a subset of $\mathcal{J}$. Assume that the
following holds:
$$\forall j\in \mathcal{J},\exists \kappa\in 2^{\mathcal{J}}, j\in\kappa~~~ and ~~~\mathbb{P}(\zeta_{1}=\kappa)>0.\eqno{(2.1)}$$
\end{defn}

\begin{lem}
(Theorem 3 of [2]). Let $T: \mathcal{V}\rightarrow\mathcal{V}$ be
$\alpha$-averaged and Fix(T)$\neq\emptyset$. Let
$(\zeta^{k})_{k\in\mathbb{N}^{\ast}}$ be a random i.i.d. sequence on
$2^{\mathcal{J}}$ such that Condition (2.1) holds. If for all $k$,
sequence $(\beta_{k})_{k\in\mathbb{N}}$ satisfies
$$0<\liminf_{k}\beta_{k}\leq\limsup_{k}\beta_{k}<\frac{1}{\alpha}.$$
Then, almost surely, the iterated sequence

$$x^{k+1}=x^{k}+\beta_{k}(\hat{T}^{(\zeta^{k+1})}x^{k}-x^{k})\eqno{(2.2)}$$
converges to some point in Fix(T ).
\end{lem}

In particular, if  $T$ is nonexpansive,  and for all $k$, sequence
$(\beta_{k})_{k\in\mathbb{N}}$ satisfies
$$0<\liminf_{k}\beta_{k}\leq\limsup_{k}\beta_{k}<1.$$
We can know the iterated sequence (2.2) converges to some point in
Fix(T ).

\section{Splitting primal-dual fixed
point algorithm }

When $g=0$, for problem (1.1) Chen et al [1] considered a
primal-dual fixed point algorithm based on the proximity
operator($PDFP^{2}O$) as follows:

$$
\left\{
\begin{array}{l}
v^{k+1}=(I-prox_{\frac{\gamma}{\lambda}h})(D(x^{k}-\gamma\nabla f(x^{k}))+(I-\lambda DD^{T})v^{k}),\\
x^{k+1}=x^{k}-\gamma\nabla f(x^{k})-\lambda D^{T}v^{k+1},
\end{array}
\right.\eqno{(3.1)}
$$
where $0<\gamma<2\beta$, $0<\lambda\leq1/\lambda_{\max}(DD^{T})$,
$\lambda_{\max}(DD^{T})$ is the largest eigenvalue of $DD^{T}$, $I$
is identity operator or unit matrix. \\
The convergence of $PDFP^{2}O$ is guaranteed by the following
theorem.
\begin{thm}([1])
 Suppose $0 <\gamma< 2\beta$ and $0<\lambda\leq1/\lambda_{\max}(DD^{T})$. Let $u_{k} = (v_{k}, x_{k})$ be
the sequence generated by $PDFP^{2}O$. Then the sequence $\{x_{k}\}$
converges to a solution of problem (1.1).
\end{thm}
Similar to the primal-dual fixed point algorithm based on proximity
operator(PDFP$^{2}$O), we proposed an algorithm called SPDFP$^{2}$O
to solve (1.1) as follows:
\begin{algorithm}[H]
\caption{Splitting primal-dual fixed points algorithm based on
proximity operator(SPDFP$^{2}$O).}
\begin{algorithmic}\label{1}
\STATE Initialization: Choose $x^{0}, y^{0}\in \mathcal{X}$,
$v^{0}\in
\mathcal{Y}$, $0<\lambda\leq1/(\lambda_{\max}(DD^{T})+1)$, $0<\gamma<2\beta$.\\
Iterations ($k\geq0$): Update $x^{k}$, $v^{k}$, $x^{k+\frac{1}{2}}$
as follows
$$
\left\{
\begin{array}{l}
x^{k+\frac{1}{2}}=x^{k}-\gamma_\nabla f(x^{k}),\\
v^{k+1}=(I-prox_{\frac{\gamma}{\lambda}h})(Dx^{k+\frac{1}{2}}+(I-\lambda DD^{T})v^{k}-\lambda Dy^{k}),\\
y^{k+1}=(I-prox_{\frac{\gamma}{\lambda}g})(x_{k+\frac{1}{2}}+(I-\lambda )y^{k}-\lambda D^{T}v^{k}),\\
x^{k+1}=x_{k+\frac{1}{2}}-\lambda D^{T}v^{k+1}-\lambda y^{k+1}.
\end{array}
\right.
$$
end for
\end{algorithmic}
\end{algorithm}
\begin{thm}
 Suppose $0 <\gamma< 2\beta$ and $0<\lambda\leq1/(\lambda_{\max}(DD^{T})+1)$. Let $u_{k} = (v_{k}, x_{k})$ be
the sequence generated by $SPDFP^{2}O$. Then the sequence
$\{x_{k}\}$ converges to a solution of problem (1.1).
\end{thm}
\begin{proof}
By setting $\tilde{D}=(D,I)^{T}$,
 $\tilde{h}(v,y)=h(v)+g(y), \forall (v,y)\in
 \mathcal{Y}\times\mathcal{X}$, we have
 $(\tilde{h}\circ\tilde{D})(x)=h(Dx)+g(x), \forall x\in\mathcal{X}$.
So, the problem (1.1) can be formulated as  follows:
$$\min_{x\in\mathcal{X}} f(x)+(\tilde{h}\circ\tilde{D})(x),\eqno{(3.2)}$$
Based on the reference[1], we can obtain the following iterative
sequence:
$$
\left\{
\begin{array}{l}
\tilde{v}^{k+1}=(I-prox_{\frac{\gamma}{\lambda}\tilde{h}})(\tilde{D}(x^{k}-\gamma\nabla f(x^{k}))+(I-\lambda \tilde{D}\tilde{D}^{T})\tilde{v}^{k}),\\
x^{k+1}=x^{k}-\gamma\nabla f(x^{k})-\lambda
\tilde{D}^{T}\tilde{v}^{k+1},
\end{array}
\right.\eqno{(3.3)}
$$
where $0<\gamma<2\beta$,
$0<\lambda\leq1/(\lambda_{\max}\tilde{D}\tilde{D}^{T})=1/(\lambda_{\max}(DD^{T})+1)$,
$\tilde{v}_{k}=(v_{k},y_{k})^{T}$. Since the function $\tilde{h}$ is
separable with the variables $v,y$, then the formula (3.3) is
equivalent to

$$
\left\{
\begin{array}{l}
v^{k+1}=(I-prox_{\frac{\gamma}{\lambda}h})(D(x^{k}-\gamma\nabla f(x^{k}))+(I-\lambda DD^{T})v^{k}-\lambda Dy^{k}),(3.4a)\\
y^{k+1}=(I-prox_{\frac{\gamma}{\lambda}g})((x^{k}-\gamma\nabla f(x^{k}))+(I-\lambda )y^{k}-\lambda D^{T}v^{k}),~~~~~~~~(3.4b)\\
x^{k+1}=x^{k}-\gamma\nabla f(x^{k})-\lambda D^{T}v^{k+1}-\lambda
y^{k+1}.~~~~~~~~~~~~~~~~~~~~~~~~~~~~~~~~(3.4c)
\end{array}
\right.
$$

From the formula (3.3), we can easy obtain algorithm 1. So, the
above algorithm is equivalent to apply directly PDFP$^{2}$O of [1]
to solve (3.2). According to the Theorem 3.1, we can obtain the
convergence of Algorithm 1(SPDFP$^{2}$O).
\end{proof}
Furthermore,  we can analyze the convergence rate of Algorithm
1(SPDFP$^{2}$O).\\
 Let $u^{k}=(v^{k},y^{k},x^{k})$ be a sequence
obtained by algorithm SPDFP$^{2}$O. Then the sequence $u_{k}$ must
converge to a point $u^{\ast}=(v^{\ast},y^{\ast},x^{\ast})$, with
$x^{\ast}$ is a solution of problem (1.1), by the Theorem 3.7 of
[1], we know the following estimate
$$\|x^{k}-x^{\ast}\|\leq\frac{c\theta^{k}}{1-\theta},$$
where $c=\|u^{1}-u^{0}\|_{\lambda}$,
$\eta=\max\{\eta_{1},\eta_{1}\}$, with $\eta_{1}$ and $\eta_{2}$
given in condition 3.1 of [1].

From Lemma 2.2, the SPDFP$^{2}$O iterates are generated by the
action of a nonexpansive operator. Lemma 2.3 shows then that a
stochastic coordinate descent version of the  SPDFP$^{2}$O converges
towards a primal-dual point. This result will be exploited in two
directions: first, we describe a stochastic minibatch algorithm,
where a large dataset is randomly split into smaller chunks. Second,
we develop an asynchronous version of the  SPDFP$^{2}$O in the
context where it is distributed on a graph.

\section{Application to stochastic approximation}
\subsection{Problem setting}
Given an integer $N > 1$, consider the problem of minimizing a sum
of composite functions

$$\inf_{x\in\mathcal{X}}\sum_{n=1}^{N}(f_{n}(x)+g_{n}(x)),\eqno{(4.1)}$$
where we make the following assumption:

\begin{ass}
For each $n = 1, ...,N$,\\
 (1)$f_{n}$ is a convex differentiable function on $\mathcal{X}$, and its
gradient $\nabla f_{n}$ is  $1/\beta$-Lipschitz continuous on
$\mathcal{X}$ for some $\beta\in(0,+\infty)$;\\
 (2)$g_{n}\in \Gamma_{0}(\mathcal{X})$;\\
 (3) The infimum of Problem (4.1) is attained;\\
 (4) $\cap_{n=1}^{N}ridom g_{n}\neq0.$
\end{ass}

This problem arises for instance in large-scale learning
applications where the learning set is too large to be handled as a
single block. Stochastic minibatch approaches consist in splitting
the data set into $N$ chunks and to process each chunk in some
order, one at a time. The quantity $f_{n}(x) + g_{n}(x)$ measures
the inadequacy between the model (represented by parameter $x$) and
the $n$-th chunk of data. Typically, $f_{n}$ stands for a data
fitting term whereas $g_{n}$ is a regularization term which
penalizes the occurrence of erratic solutions. As an example, the
case where $f_{n}$ is quadratic and $g_{n}$ is the $l_{1}$-norm
reduces to the popular LASSO problem [5]. In particular,  it also
useful to recover sparse signal.

\subsection{ Instantiating the SPDFP$^{2}$O}
We regard our stochastic minibatch algorithm as an instance of the
SPDFP$^{2}$O coupled with a randomized coordinate descent. In order
to end that ,we rephrase problem (4.1) as
$$\inf_{x\in\mathcal{X}^{N}}\sum_{n=1}^{N}(f_{n}(x)+g_{n}(x))+\iota_{\mathcal{C}}(x),\eqno{(4.2)}$$
where the notation $x_{n}$ represents the $n$-th component of any $x
\in \mathcal{X}^{N}$, $\mathcal{C}$ is the space of vectors $x \in
\mathcal{X}^{N}$ such that $x_{1} = \cdots = x_{N}$. On the space
$\mathcal{X}^{N}$, we set $f(x) = \sum_{ n} f_{n}(x_{n})$, $g(x) =
\sum_{ n} g_{n}(x_{n})$, $h(x) = \iota_{\mathcal{C}}$ and $D =
I_{\mathcal{X}^{N}}$ the identity matrix. problem (4.2) is
equivalent to
$$\min_{x\in\mathcal{X}^{N}} f(x)+g(x)+ (h\circ D)(x).\eqno{(4.3)}$$
We define the natural scalar product on $\mathcal{X}^{N}$ as
$\langle x,y\rangle=\sum_{n=1}^{N}\langle x_{n},y_{n}\rangle$.
Applying the SPDFP$^{2}$O to solve problem (4.3) leads to the
following iterative scheme:
\begin{align*}
&z^{k+1}=proj_{\mathcal{C}}(x^{k}-\gamma\nabla f(x^{k})+(1-\lambda)v^{k}-\lambda y^{k}), \\
&v^{k+1}_{n}=x^{k}_{n}-\gamma\nabla f_{n}(x^{k}_{n})+(1-\lambda)v^{k}_{n}-\lambda y^{k}_{n}-z^{k+1}_{n}, \\
&y^{k+1}_{n}=(I-prox_{\frac{\gamma}{\lambda}g_{n}})(x^{k}_{n}-\gamma\nabla f_{n}(x^{k}_{n})+(1-\lambda)y^{k}_{n}-\lambda v^{k}_{n}),\\
&x^{k+1}_{n}=x^{k}_{n}-\gamma\nabla f_{n}(x^{k}_{n})-\lambda
v^{k+1}_{n}-\lambda y^{k+1}_{n},
\end{align*}
where $proj_{\mathcal{C}}$ is the orthogonal projection onto
$\mathcal{C}$. Observe that for any $x \in \mathcal{X}^{N}$,
$proj_{\mathcal{C}}(x)$ is equivalent to $(\bar{x}, \cdots ,
\bar{x})$ where $\bar{x}$ is the average of vector $x$, that is
$\bar{x}=N^{-1}\sum_{n}x_{n}$. Consequently, the components of
$z^{k+1}$ are equal and coincide with $\bar{x}^{k}-\gamma\nabla
\bar{f}(\bar{x}^{k})+(1-\lambda)\bar{v}^{k}-\lambda \bar{y}^{k}$
where $\bar{f}$, $\bar{x}^{k}$, $\bar{v}^{k}$ and $\bar{y}^{k}$ are
the averages of $f$, $x^{k}$ $v^{k}$ and $y^{k}$ respectively. By
inspecting the $v^{k}$ $n$-update equation above, we notice that the
latter equality simplifies even further by noting that
$\bar{v}^{k+1} = 0$ or, equivalently, $\bar{v}^{k} = 0$ for all
$k\geq 1$ if the algorithm is started with $\bar{v}^{0} = 0$.
Finally, for any $n$ and $k\geq 1$, the above iterations reduce to
\begin{align*}
&\bar{x}^{k}-\gamma\nabla \bar{f}(\bar{x}^{k})-\lambda \bar{y}^{k}=\frac{1}{N}\sum_{n=1}^{N} (x^{k}_{n}-\gamma\nabla \bar{f}_{n}(x^{k}_{n})-\lambda y^{k}_{n}),\\
&v^{k+1}_{n}=x^{k}_{n}-\gamma\nabla f_{n}(x^{k}_{n})+(1-\lambda)v^{k}_{n}-\lambda y^{k}_{n}-(\bar{x}^{k}-\gamma\nabla \bar{f}(\bar{x}^{k})-\lambda \bar{y}^{k}),\\
&y^{k+1}_{n}=(I-prox_{\frac{\gamma}{\lambda}g_{n}})(x^{k}_{n}-\gamma\nabla f_{n}(x^{k}_{n})+(1-\lambda)y^{k}_{n}-\lambda v^{k}_{n}),\\
&x^{k+1}_{n}=x^{k}_{n}-\gamma\nabla f_{n}(x^{k}_{n})-\lambda
v^{k+1}_{n}-\lambda y^{k+1}_{n}.
\end{align*}
These iterations can be written more compactly as
\begin{algorithm}[H]
\caption{Minibatch SPDFP$^{2}$O.}
\begin{algorithmic}\label{1}
\STATE Initialization: Choose $x^{0}, y^{0}\in \mathcal{X}$,
$v^{0}\in
\mathcal{Y}$, s.t. $\sum_{n}v^{0}_{n}=0$ ,  $0<\lambda\leq1/2$,  $0<\gamma<2\beta$.\\
Do
$$
\begin{array}{l}
\bullet ~~\bar{x}^{k}-\gamma\nabla \bar{f}(\bar{x}^{k})-\lambda \bar{y}^{k}=\frac{1}{N}\sum_{n=1}^{N} (x^{k}_{n}-\gamma\nabla f_{n}(x^{k}_{n})-\lambda y^{k}_{n}),\\
\bullet ~~For~~~  batches~~~  n = 1, \cdots ,N,~~~  do\\
~~~~v^{k+1}_{n}=x^{k}_{n}-\gamma\nabla f_{n}(x^{k}_{n})+(1-\lambda)v^{k}_{n}-\lambda y^{k}_{n}-(\bar{x}^{k}-\gamma\nabla \bar{f}(\bar{x}^{k})-\lambda \bar{y}^{k}),\\
~~~~y^{k+1}_{n}=(I-prox_{\frac{\gamma}{\lambda}g_{n}})(x^{k}_{n}-\gamma\nabla f_{n}(x^{k}_{n})+(1-\lambda)y^{k}_{n}-\lambda v^{k}_{n}),\\
~~~~x^{k+1}_{n}=x^{k}_{n}-\gamma\nabla f_{n}(x^{k}_{n})-\lambda
v^{k+1}_{n}-\lambda y^{k+1}_{n}.\\
\bullet ~~Increment~ k.
\end{array}\eqno{(4.4)}
$$

\end{algorithmic}
\end{algorithm}
The following result is a straightforward consequence of Theorem
3.2.
\begin{thm}
 Suppose $0 <\gamma< 2\beta$ and $0<\lambda\leq1/2$, and let Assumption 4.1 hold true.   Then for any
initial point $( v^{0}, y^{0}, x^{0})$ such that $\bar{v}^{0} = 0$,
the sequence $\{\bar{x}^{k}\}$ generated by Minibatch SPDFP$^{2}$O
converges to a solution of problem (4.3).
\end{thm}
At each step $k$, the iterations given above involve the whole set
of functions $f_{n}, g_{n} (n = 1, \cdots,N)$. Our aim is now to
propose an algorithm which involves a single couple of functions
$(f_{n}, g_{n})$ per iteration.
\subsection{ A stochastic minibatch splitting primal-dual fixed
point algorithm}
We are now in position to state the main algorithm
of this section. The proposed stochastic minibatch splitting
primal-dual fixed point algorithm(SMSPDFP$^{2}$O) is obtained upon
applying the randomized coordinate descent on the minibatch
SPDFP$^{2}$O:

\begin{algorithm}[H]
\caption{SMSPDFP$^{2}$O.}
\begin{algorithmic}\label{1}
\STATE Initialization: Choose $x^{0}, y^{0}\in \mathcal{X}$,
$v^{0}\in
\mathcal{Y}$,  $0<\lambda\leq1/2$,  $0<\gamma<2\beta$.\\
Do
$$
\begin{array}{l}
\bullet ~~Define ~~ \bar{x}^{k}-\gamma\nabla \bar{f}(\bar{x}^{k})-\lambda \bar{y}^{k}=\frac{1}{N}\sum_{n=1}^{N} (x^{k}_{n}-\gamma\nabla f_{n}(x^{k}_{n})-\lambda y^{k}_{n}),\\
 ~~~~ ~~~~~~~~~~~\bar{v}^{k}=\frac{1}{N}\sum_{n=1}^{N}v^{k}_{n},\\
\bullet ~~Pick ~up ~the ~value ~of \zeta^{k+1}, \\
\bullet ~~For~  batch~n=\zeta^{k+1},~set\\
~~~~v^{k+1}_{n}=x^{k}_{n}-\gamma\nabla f_{n}(x^{k}_{n})+(1-\lambda)v^{k}_{n}-\lambda y^{k}_{n}-(\bar{x}^{k}-\gamma\nabla \bar{f}(\bar{x}^{k})-(1-\lambda)\bar{v}^{k}-\lambda \bar{y}^{k}),\\
~~~~~~~~~~~~~~~~~~~~~~~~~~~~~~~~~~~~~~~~~~~~~~~~~~~~~~~~~~~~~~~~~~~~~~~~~~~~~~~~~~~~~~~~~~~~~~~~~~~~(4.5a)\\
~~~~y^{k+1}_{n}=(I-prox_{\frac{\gamma}{\lambda}g_{n}})(x^{k}_{n}-\gamma\nabla f_{n}(x^{k}_{n})+(1-\lambda)y^{k}_{n}-\lambda v^{k}_{n}),~~~~~~~~~~~~~~~~~~~~~(4.5b)\\
~~~~x^{k+1}_{n}=x^{k}_{n}-\gamma\nabla f_{n}(x^{k}_{n})-\lambda
v^{k+1}_{n}-\lambda y^{k+1}_{n}.~~~~~~~~~~~~~~~~~~~~~~~~~~~~~~~~~~~~~~~~~~~~(4.5c)\\
\bullet ~~For~ all ~batches ~n\neq\zeta^{k+1},~~ v^{k+1}_{n}=v^{k}_{n}, y^{k+1}_{n}=y^{k}_{n}, x^{k+1}_{n}=x^{k}_{n}.\\
 \bullet ~~Increment~ k.
\end{array}
$$

\end{algorithmic}
\end{algorithm}

\begin{ass}
The random sequence $(\zeta^{k})_{k\in\mathbb{N}^{\ast}}$ is i.i.d.
and satisfies $\mathbb{P}[\zeta^{1} = n] > 0$ for all $n = 1,
...,N$.
\end{ass}

\begin{thm}
 Suppose $0 <\gamma< 2\beta$ and $0<\lambda\leq1/2$, and let Assumption 4.1 and 4.2 hold true.   Then for any initial point
$( v^{0}, y^{0}, x^{0})$ , the sequence $\{\bar{x}^{k}\}$ generated
by SMSPDFP$^{2}$O converges to a solution of problem (4.3).
\end{thm}
\begin{proof}
Let us define the functions $f$, $g$, and $h$ are the ones defined
in Section 4.2 and $D=I_{\mathcal{X}^{N}}$. Then the iterates
$((v^{k+1}_{ n} )^{N} _{n=1}, (y^{k+1}_{ n} )^{N} _{n=1}, (x^{k+1}_{
n} )^{N} _{n=1})$ described by Equations (4.4) coincide with the
iterates $(v^{k+1}, y^{k+1}, x^{k+1})$ described by Equations (3.4).
If we write these equations more compactly as $(\tilde{v}^{k+1},
x^{k+1}) = T(\tilde{v}^{k}, x^{k})$ where
$\tilde{v}^{k}=(v^{k},y^{k})^{T}$ , the operator $T$ acts in the
space
$\mathcal{V}=\mathcal{X}^{N}\times\mathcal{X}^{N}\times\mathcal{X}^{N}$,
then Lemma 2.2 shows that $T$ is nonexpansive. Defining the
selection operator $\mathcal{S}_{n }$ on $\mathcal{V}$ as
$\mathcal{S}_{n }(\tilde{v}, x) = (\tilde{v}_{n}, x_{n})$, we obtain
that $\mathcal{V} = \mathcal{S}_{1 }(\mathcal{V})\times \cdots\times
\mathcal{S}_{N }(\mathcal{V})$ up to an element reordering. To be
compatible with the notations of Definition 2.4, we assume that $J =
N$ and that the random sequence $\zeta^{k}$ driving the
SMSPDFP$^{2}$O algorithm is set valued in
$\{\{1\},\ldots\{N\}\}\subset2^{\mathcal{J}}$. In order to establish
Theorem 4.2, we need to show that the iterates $(\tilde{v}^{k+1},
x^{k+1})$ provided by the SMSPDFP$^{2}$O algorithm are those who
satisfy the equation $(\tilde{v}^{k+1}, x^{k+1}) =
T^{(\zeta^{k+1})}(\tilde{v}^{k}, x^{k})$.
 By the direct application of Lemma 2.3, we can obtain Theorem 4.2.
If we write $(\tilde{\delta}^{k+1}, \sigma^{k+1}) = T(\tilde{v}^{k},
x^{k})$ where $\tilde{\delta}^{k+1}=(\mu^{k+1},\nu^{k+1})^{T}$, then
by Eq. (3.4a),

$$\mu^{k+1}_{n}=x^{k}_{n}-\gamma\nabla f_{n}(x^{k}_{n})+(1-\lambda)v^{k}_{n}-\lambda y^{k}_{n}-(\bar{x}^{k}-\gamma\nabla f(\bar{x}^{k})-(1-\lambda)\bar{v}^{k}-\lambda \bar{y}^{k})n=1,\ldots N.$$
Observe that in general, $\bar{v}^{k}\neq 0$ because in the
SMSPDFP$^{2}$O algorithm, only one component is updated at a time.
If $\{n\} = \zeta^{k+1}$, then $v^{k+1}_{n} = \mu^{k+1}_{n}$ which
is Eq. (4.5a). All other components of $v^{k}$ are carried over to $
v^{k+1}$ .\\
 By Equation (3.4b) and (3.4c) we also get
$$\nu^{k+1}_{n}=(I-prox_{\frac{\gamma}{\lambda}g_{n}})(x^{k}_{n}-\gamma\nabla f_{n}(x^{k}_{n})+(1-\lambda)y^{k}_{n}-\lambda v^{k}_{n}),$$
$$\sigma^{k+1}_{n}=x^{k}_{n}-\gamma\nabla f_{n}(x^{k}_{n})-\lambda
v^{k+1}_{n}-\lambda y^{k+1}_{n}.$$ If $\{n\} = \zeta^{k+1}$, then
$y^{k+1}_{n}=\nu^{k+1}_{n}$, $x^{k+1}_{n}=\sigma^{k+1}_{n}$ can
easily be shown to be given by (4.5b) and (4.5c).

\end{proof}

\section{Distributed optimization}

~~~~Consider a set of $N > 1$ computing agents that cooperate to
solve the minimization problem (4.1). Here, $f_{n}$, $g_{n}$ are two
private functions available at agent $n$. Our purpose is to
introduce a random distributed algorithm to  solve (4.1). The
algorithm is asynchronous in the sense that some components of the
network are allowed to wake up at random and perform local updates,
while the rest of the network stands still. No coordinator or global
clock is needed. The frequency of activation of the various network
components is likely to vary.

The examples of this problem appear in learning applications where
massive training data sets are distributed over a network and
processed by distinct machines [5], [6], in resource allocation
problems for communication networks [7], or in statistical
estimation problems by sensor networks [8], [9].
\subsection{Network model and problem formulation}
We consider the network as a graph $G = (Q,E)$ where $Q = \{1,
\cdots ,N\}$ is the set of agents/nodes and $E\subset \{1, \cdots
,N\}^{2}$ is the set of undirected edges. We write $n\sim m$
whenever ${n,m}\in E$. Practically, $n\sim m $ means that agents $n$
and $m$ can communicate with each other.

\begin{ass}
$G$ is connected and has no self loop.
\end{ass}

Now we introduce some notations. For any $x\in\mathcal{X}^{|Q|}$, we
denote by $x_{n}$ the components of $x$, i.e., $x = (x_{n})_{n\in
Q}$. We redard the functions $f$ and $g$ on
$\mathcal{X}^{|Q|}\rightarrow(-\infty,+\infty]$ as $f(x)=\sum_{n\in
Q}f_{n}(x_{n})$ and $g(x)=\sum_{n\in Q}g_{n}(x_{n})$. So the problem
(4.1) is equal to the minimization of $f(x)+g(x)$ under the
constraint that all components of $x$ are equal.

Next we write the latter constraint in a way that involves the graph
$G$. We replace the global consensus constraint by a modified
version of the function $\iota_{\mathcal{C}}$ . The purpose of us is
to ensure global consensus through local consensus over every edge
of the graph.

For any $\varepsilon\in E$, say $\varepsilon= \{n,m\}\in Q$ , we
define the linear operator $D_{\varepsilon}(x) : \mathcal{X}^{|Q|}
\rightarrow \mathcal{X}^{2}$ as $D_{\varepsilon}(x) = (x_{n},
x_{m})$ where we assume some ordering on the nodes to avoid any
ambiguity on the definition of $D$. We construct the linear operator
$D:\mathcal{X}^{|Q|} \rightarrow\mathcal{Y}\triangleq
\mathcal{X}^{2|Q|}$ as $D(x)=(D_{\varepsilon}(x))_{\varepsilon\in
E}$ where we also assume some ordering on the edges. Any vector $y
\in \mathcal{Y}$ will be written as $y =
(y_{\varepsilon})_{\varepsilon\in E}$ where, writing $\varepsilon=
\{n,m\} \in E$, the component $y_{\varepsilon}$ will be represented
by the couple $y_{\varepsilon}= (y_{\varepsilon}(n),
y_{\varepsilon}(m))$ with $n < m$. We also introduce the subspace of
 $\mathcal{X}^{2}$ defined as $\mathcal{C}_{2} = \{(x, x) : x \in \mathcal{X}\}$. Finally, we define $h : \mathcal{Y} \rightarrow (
-\infty,+\infty]$ as

$$h(y)=\sum_{\varepsilon\in E}\iota_{\mathcal{C}_{2}}(y_{\varepsilon}).\eqno{(5.1)}$$
Then we consider the following problem:
$$\min_{x\in\mathcal{X}^{|Q|}} f(x)+g(x)+ (h\circ D)(x).\eqno{(5.2)}$$

\begin{lem}
([2]). Let Assumptions 5.1 hold true.  The minimizers of (5.2) are
the tuples $(x^{\ast}, \cdots , x^{\ast})$ where $x^{\ast}$ is any
minimizer of (4.1).
\end{lem}

\subsection{Instantiating the SPDFP$^{2}$O}

Now we use the SPDFP$^{2}$O to solve the problem (5.2). Since the
newly defined function $h$ is separable with respect to the
$(y_{\varepsilon})_{\varepsilon\in E}$, we get

$$prox_{\tau h}(y)=(prox_{\tau \iota_{\mathcal{C}_{2}}}(y_{\varepsilon}))_{\varepsilon\in E}=((\bar{y}_{\varepsilon}, \bar{y}_{\varepsilon}))_{\varepsilon\in E},$$
where
$\bar{y}_{\varepsilon}=(y_{\varepsilon}(n)+y_{\varepsilon}(m))/2$ if
$\varepsilon=\{n,m\}$. With this at hand, the update equation (3.4a)
of the SPDFP$^{2}$O can be  written as
$$z^{k+1}=((\bar{z}_{\varepsilon}^{k+1},\bar{z}_{\varepsilon}^{k+1}))_{\varepsilon\in E},$$
where
\begin{align*}
\bar{z}^{k+1}&=\frac{x^{k}_{n}-\gamma\nabla f_{n}(x^{k}_{n})-\lambda y^{k}_{n}-\lambda d_{n}v_{n}^{k}+x^{k}_{m}-\gamma\nabla f_{m}(x^{k}_{m})-\lambda y^{k}_{m}-\lambda d_{m}v_{m}^{k}}{2}\\
&+\frac{v_{\varepsilon}^{k}(n)+v_{\varepsilon}^{k}(m)}{2},
\end{align*}
 for any $\varepsilon= \{n,m\}\in E$. and $ d_{n}x_{n}$ coincides with the $n$-th component of the vector
$D^{T}Dx$
 , $ d_{n}$ is the degree (i.e., the number of
neighbors) of node $n$. Plugging this equality into Eq. (3.4a), it
can be seen that $v_{\varepsilon}^{k}(n)=-v_{\varepsilon}^{k}(m)$.
Therefore
$$\bar{z}^{k+1}=\frac{x^{k}_{n}-\gamma\nabla f_{n}(x^{k}_{n})-\lambda y^{k}_{n}-\lambda d_{n}v_{n}^{k}+x^{k}_{m}-\gamma\nabla f_{m}(x^{k}_{m})-\lambda y^{k}_{m}-\lambda d_{m}v_{m}^{k}}{2},$$
for any $k \geq 1$. Moreover
$$v_{\varepsilon}^{k+1}=\frac{x^{k}_{n}-\gamma\nabla f_{n}(x^{k}_{n})-\lambda y^{k}_{n}-\lambda d_{n}v_{n}^{k}-(x^{k}_{m}-\gamma\nabla f_{m}(x^{k}_{m})-\lambda y^{k}_{m}-\lambda d_{m}v_{m}^{k})}{2}+v_{\varepsilon}^{k}(n).$$
 From (3.4b) and (3.4c), the $n^{th}$
component of $y^{k+1}$ and $  x^{k+1}$ can be written
$$y^{k+1}_{n}=(I-prox_{\frac{\gamma}{\lambda}g_{n}})(x^{k}_{n}-\gamma\nabla f_{n}(x^{k}_{n})+(1-\lambda)y^{k}_{n}-\lambda D^{T}v^{k}_{n}),$$
$$x^{k+1}_{n}=x^{k}_{n}-\gamma\nabla f_{n}(x^{k}_{n})-\lambda D^{T}
v^{k+1}_{n}-\lambda y^{k+1}_{n},$$ where for any $v\in\mathcal{Y}$,

$$(D^{T}v)_{n}=\sum_{m:\{n,m\}\in E}v_{\{n,m\}}(n)$$
is the $n$-th component of $D^{T}v \in \mathcal{X}^{|Q |}$. Plugging
Eq. (3.4b) and (3.4c) together with the expressions of
$\bar{z}^{k+1}_{\{n,m\}}$ and $v^{k+1}_{\{n,m\}}$, we can have
$$y^{k+1}_{n}=(I-prox_{\frac{\gamma}{\lambda}g_{n}})(x^{k}_{n}-\gamma\nabla f_{n}(x^{k}_{n})+(1-\lambda)y^{k}_{n}-\lambda \sum_{m:\{n,m\}\in E}v_{\{n,m\}}^{k}(n)),$$

\begin{align*}
x^{k+1}_{n}&=x^{k}_{n}-\gamma\nabla f_{n}(x^{k}_{n})-\lambda
\sum_{m:\{n,m\}\in
E}v_{\{n,m\}}^{k}(n) -\lambda y^{k+1}_{n}\\
&-\lambda \sum_{m:\{n,m\}\in E}(\frac{x^{k}_{n}-\gamma\nabla
f_{n}(x^{k}_{n})-\lambda y^{k}_{n}-\lambda
d_{n}v_{n}^{k}-(x^{k}_{m}-\gamma\nabla f_{m}(x^{k}_{m})-\lambda
y^{k}_{m}-\lambda d_{m}v_{m}^{k})}{2}).
\end{align*}
The algorithm is finally described by the following procedure: Prior
to the clock tick $k + 1$, the node $n$ has in its memory the
variables $x^{k}_{n}$, $y^{k}_{n}$,
$\{v_{\{n,m\}}^{k}(n)\}_{m\thicksim n}$, $\{x^{k}_{m}\}_{m\thicksim
n}$ and  $\{y^{k}_{m}\}_{m\thicksim n}$.
\begin{algorithm}[H]
\caption{Distributed SPDFP$^{2}$O.}
\begin{algorithmic}\label{1}
\STATE Initialization: Choose $x^{0}, y^{0}\in \mathcal{X}$,
$v^{0}\in
\mathcal{Y}$, s.t. $\sum_{n}v^{0}_{n}=0$ , $0<\lambda\leq1/(\lambda_{\max}(DD^{T})+1)$,  $0<\gamma<2\beta$.\\
Do
$$
\begin{array}{l}
\bullet ~~For~ any~ n \in Q ,~ Agent ~n ~performs ~the~ following~
operations: \\
~~~~v_{\{n,m\}}^{k+1}(n)=\frac{x^{k}_{n}-\gamma\nabla f_{n}(x^{k}_{n})-\lambda y^{k}_{n}-\lambda d_{n}v_{n}^{k}-(x^{k}_{m}-\gamma\nabla f_{m}(x^{k}_{m})-\lambda y^{k}_{m}-\lambda d_{m}v_{m}^{k})}{2}+v_{\{n,m\}}^{k}(n),\\
~~~~~~~~~~~~~~~~~~~~for~all~m\thicksim n,~~~~~~~~~~~~~~~~~~~~~~~~~~~~~~~~~~~~~~~~~~~~~~~~~~~~~~~~~~~(5.3a)\\
~~~~~~~~~~y^{k+1}_{n}=(I-prox_{\frac{\gamma}{\lambda}g_{n}})(x^{k}_{n}-\gamma\nabla f_{n}(x^{k}_{n})+(1-\lambda)y^{k}_{n}\\
~~~~~~~~~~~~~~~~~~-\lambda \sum_{m:\{n,m\}\in E}v_{\{n,m\}}^{k}(n)),~~~~~~~~~~~~~~~~~~~~~~~~~~~~~~~~~~~~~~~~~~~~~~(5.3b)\\
~~~~~~~~~~x^{k+1}_{n}=x^{k}_{n}-\gamma\nabla
f_{n}(x^{k}_{n})-\lambda \sum_{m:\{n,m\}\in
E}v_{\{n,m\}}^{k}(n) -\lambda y^{k+1}_{n}\\
~~~~~~~~~~~~~~~~~~-\lambda \sum_{m:\{n,m\}\in
E}(\frac{x^{k}_{n}-\gamma\nabla f_{n}(x^{k}_{n})-\lambda
y^{k}_{n}-\lambda d_{n}v_{n}^{k}-(x^{k}_{m}-\gamma\nabla
f_{m}(x^{k}_{m})-\lambda
y^{k}_{m}-\lambda d_{m}v_{m}^{k})}{2}).\\
~~~~~~~~~~~~~~~~~~~~~~~~~~~~~~~~~~~~~~~~~~~~~~~~~~~~~~~~~~~~~~~~~~~~~~~~~~~~~~~~~~~~~~~~~~~~~~~~~~~~(5.3c)\\
 \bullet ~~Agent~ n ~sends~ the~ parameter~ y^{k+1}_{n},
x^{k+1}_{n}
 ~to~ their~ neighbors~ respectively.\\
 \bullet ~~Increment~ k.
\end{array}
$$

\end{algorithmic}
\end{algorithm}

\begin{thm}
 Suppose $0 <\gamma< 2\beta$ and $0<\lambda\leq1/(\lambda_{\max}(DD^{T})+1)$, and let Assumption 4.1 and 5.1 hold true.  Let $u^{k} = (v^{k},y^{k}, x^{k})$ be
the sequence generated by Distributed SPDFP$^{2}$O for any initial
point $( v^{0}, y^{0}, x^{0})$ . Then for all $n \in Q$ the sequence
$(x^{k}_{n})_{k\in \mathbb{N}}$ converges to a solution of problem
(4.1).
\end{thm}

\subsection{A Distributed asynchronous splitting primal-dual fixed
point algorithm} In this section, we use the randomized coordinate
descent on the above algorithm, we call this algorithm as
distributed asynchronous splitting primal-dual fixed point algorithm
(DSSPDFP$^{2}$O). This algorithm has the following attractive
property: at each iteration, a single agent, or possibly a subset of
agents chosen at random, are activated. Moreover, if we let
$(\zeta^{k})_{k\in \mathbb{N}}$ be a sequence of i.i.d. random
variables valued in $2^{ Q}$. The value taken by $\zeta^{k}$
represents the agents that will be activated and perform a prox on
their $x$ variable at moment $k$. The asynchronous algorithm goes as
follows:
\begin{algorithm}[H]
\caption{DASPDFP$^{2}$O.}
\begin{algorithmic}\label{1}
\STATE Initialization: Choose $x^{0}, y^{0}\in \mathcal{X}$,
$v^{0}\in
\mathcal{Y}$,  $0<\lambda\leq1/(\lambda_{\max}(DD^{T})+1)$,  $0<\gamma<2\beta$.\\
Do
$$
\begin{array}{l}
\bullet ~~Select~ a~ random~ set ~of~ agents~ \zeta^{k+1} =\mathcal{B}. \\
\bullet ~~For~ any~ n \in \mathcal{B},~ Agent~ n~ performs~ the
~following~
operations:\\
~~~~-For ~all~ m \thicksim n, do\\
~~~~~~~~v_{\{n,m\}}^{k+1}(n)=\frac{x^{k}_{n}-\gamma\nabla f_{n}(x^{k}_{n})-\lambda y^{k}_{n}-\lambda d_{n}v_{n}^{k}-(x^{k}_{m}-\gamma\nabla f_{m}(x^{k}_{m})-\lambda y^{k}_{m}-\lambda d_{m}v_{m}^{k})}{2}\\
~~~~~~~~~~~~~~~~~~~~~~~~+\frac{v_{\{n,m\}}^{k}(n)-v_{\{n,m\}}^{k}(m)}{2},\\
~~~~-y^{k+1}_{n}=(I-prox_{\frac{\gamma}{\lambda}g_{n}})(x^{k}_{n}-\gamma\nabla f_{n}(x^{k}_{n})+(1-\lambda)y^{k}_{n}\\
~~~~~~~~~~~~~~~~~-\lambda \sum_{m:\{n,m\}\in E}v_{\{n,m\}}^{k}(n)),\\
~~~~-x^{k+1}_{n}=x^{k}_{n}-\gamma\nabla f_{n}(x^{k}_{n})+\lambda
\sum_{m:\{n\thicksim m\}\in
E}v_{\{n,m\}}^{k}(m) -\lambda y^{k+1}_{n}\\
~~~~~~~~~~~~~-\lambda \sum_{m:\{n,m\}\in
E}(\frac{x^{k}_{n}-\gamma\nabla f_{n}(x^{k}_{n})-\lambda
y^{k}_{n}-\lambda d_{n}v_{n}^{k}-(x^{k}_{m}-\gamma\nabla
f_{m}(x^{k}_{m})-\lambda
y^{k}_{m}-\lambda d_{m}v_{m}^{k})}{2}).\\
~~~~-For ~all~ m \thicksim n,~ send
\{y^{k+1}_{n},x^{k+1}_{n},v_{\{n,m\}}^{k+1}(n)\}~to~ Neighbor~
m.\\
\bullet ~~For ~any~ agent ~n\neq\mathcal{B}, ~
y^{k+1}_{n}=y^{k}_{n}, x^{k+1}_{n}=x^{k}_{n},
~and~~v_{\{n,m\}}^{k+1}(n)=~v_{\{n,m\}}^{k}(n)\\
~~~~for~ all~ m \thicksim n.\\
 \bullet ~~Increment~ k.
\end{array}
$$

\end{algorithmic}
\end{algorithm}

\begin{ass}
The collections of sets $\{\mathcal{B}_{1},\mathcal{B}_{2},\ldots\}$
such that $\mathbb{P}[\zeta^{1} = \mathcal{B}_{i}]$ is positive
satisfies $ \bigcup\mathcal{B}_{i} =Q$.
\end{ass}

\begin{thm}
 Suppose $0 <\gamma< 2\beta$ and $0<\lambda\leq1/(\lambda_{\max}(DD^{T})+1)$, and let Assumption 4.1, 5.1 and 5.2 hold true.  Let $(u^{k}_{n})_{n\in Q} = (v^{k}_{n},y^{k}_{n}, x^{k}_{n})_{n\in Q}$ be
the sequence generated by DASPDFP$^{2}$O for any initial point $(
v^{0}, y^{0}, x^{0})$ . Then  the sequence
$x^{k}_{1},\ldots,x^{k}_{|Q|}$ converges to a solution of problem
(4.1).
\end{thm}

\begin{proof}
Let us define $f$, $g$, $h$ and $D$ are the ones defined in the
problem 5.2. By Equations (3.4). We write these equations more
compactly as $(\tilde{v}^{k+1}, x^{k+1}) = T(\tilde{v}^{k}, x^{k})$
where $\tilde{v}^{k}=(v^{k},y^{k})^{T}$ , the operator $T$ acts in
the space
$\mathcal{V}=\mathcal{R}\times\mathcal{X}^{|Q|}\times\mathcal{X}^{|Q|}$,
and $\mathcal{R}$ is the image of $\mathcal{X}^{|Q|}$ by $D$. then
by Lemma 2.2 we know $T$ is nonexpansive. Defining the selection
operator $\mathcal{S}_{n }$ on $\mathcal{V}$ as $\mathcal{S}_{n
}(\tilde{v}, x) = (\tilde{v}_{\varepsilon}(n)_{\varepsilon\in
Q:n\in\varepsilon}, x_{n})$, where
$\tilde{v}_{\varepsilon}(n)_{\varepsilon\in
Q:n\in\varepsilon}=(v_{\varepsilon}(n)_{\varepsilon\in
Q:n\in\varepsilon}, y_{n})^{T}$. So, we obtain that $\mathcal{V} =
\mathcal{S}_{1 }(\mathcal{V})\times \cdots\times \mathcal{S}_{|Q|
}(\mathcal{V})$ up to an element reordering. Identifying the set
$\mathcal{J}$ introduced in the notations of Definition 2.4 with
$Q$, the operator $T^{(\zeta^{k})}$ is defined as follows:
$$
\mathcal{S}_{n }(T^{(\zeta^{k})}(\tilde{v}, x)) = \left\{
\begin{array}{l}
\mathcal{S}_{n }(T(\tilde{v}, x)),\,\,\, \,\,if \,n\,
\in\zeta^{k},\\
\mathcal{S}_{n }(\tilde{v}, x),\,\,\,\,\,\,\,\,\,\,\, \,\,if \,n\,
\neq\zeta^{k} .
\end{array}
\right.
$$
Then by Lemma 2.3, we know the sequence $(\tilde{v}^{k+1}, x^{k+1})
= T^{(\zeta^{k+1})}(\tilde{v}^{k}, x^{k})$ converges almost surely
to a solution of problem (1). Moreover, from Lemma 5.1, we have the
sequence $x^{k}$ converges almost surely to a solution of problem
(4.1).

Therefore we need to show that the operator $T^{(\zeta^{k+1})}$ is
translated into the DASPDFP$^{2}$O algorithm. If we write
$(\tilde{\delta}^{k+1}, \sigma^{k+1}) = T(\tilde{v}^{k}, x^{k})$
where $\tilde{\delta}^{k+1}=(\mu^{k+1},\nu^{k+1})^{T}$, then by Eq.
(3.4a),
\begin{align*}
\mu_{\varepsilon}^{k+1}&=\frac{x^{k}_{n}-\gamma\nabla f_{n}(x^{k}_{n})-\lambda y^{k}_{n}-\lambda d_{n}v_{n}^{k}-(x^{k}_{m}-\gamma\nabla f_{m}(x^{k}_{m})-\lambda y^{k}_{m}-\lambda d_{m}v_{m}^{k})}{2}\\
&+\frac{v_{\varepsilon}^{k}(n)+v_{\varepsilon}^{k}(m)}{2}.
\end{align*}
Getting back to $(\tilde{v}^{k+1}, x^{k+1}) =
T^{(\zeta^{k+1})}(\tilde{v}^{k}, x^{k})$, we have for all $n \in
\zeta^{k+1}$ and all $m \thicksim n$,  then $v^{k+1}_{\{n,m\}}(n) =
\mu^{k+1}_{\{n,m\}}(n)$ .
 By Equation (3.4b) and (3.4c) we also get
$$\nu^{k+1}_{n}=(I-prox_{\frac{\gamma}{\lambda}g_{n}})(x^{k}_{n}-\gamma\nabla f_{n}(x^{k}_{n})+(1-\lambda)y^{k}_{n}-\lambda D^{T}v^{k}_{n}),$$
$$\sigma^{k+1}_{n}=x^{k}_{n}-\gamma\nabla f_{n}(x^{k}_{n})-\lambda D^{T}
v^{k+1}_{n}-\lambda y^{k+1}_{n}.$$ Therefore, for all $n\in
\zeta^{k+1}$, then $y^{k+1}_{n}=\nu^{k+1}_{n}$,
$x^{k+1}_{n}=\sigma^{k+1}_{n}$. If we use the identity
$(D^{T}v)_{n}=\sum_{m:\{n,m\}\in E}v_{\{n,m\}}(n)$ on the above
equations, it can easy check these equations coincides with the
$x$-update $y$-update  in the DASPDFP$^{2}$O algorithm.

\end{proof}

\section{Numerical experiments}
We consider the problem of $l_{1}$-regularized logistic regression.
Denoting by $m$ the number of observations and by $q$ the number of
features, the optimization problem writes
$$\inf_{x\in \mathbb{R}^{q}}\frac{1}{m}\sum_{i=1}^{m}\log(1+e^{-y_{i}a_{i}^{T}x})+\tau\|x\|_{1},\eqno{(6.1)}$$
where the $(y_{i})_{i=1}^{m}$ are in $\{-1,+1\}$, the
$(a_{i})_{i=1}^{m}$ are in $\mathbb{R}^{q}$, and $\tau>0$ is a
scalar. Let $(\mathcal{W})_{n=1}^{N}$ indicate a partition of $\{1,
. . . ,m\}$. The optimization problem then writes

$$\inf_{x\in \mathbb{R}^{q}}\sum_{n=1}^{N}\sum_{i\in\mathcal{W}_{n}}\frac{1}{m}\log(1+e^{-y_{i}a_{i}^{T}x})+\tau\|x\|_{1},\eqno{(6.2)}$$

or, splitting the problem between the batches

$$\inf_{x\in \mathbb{R}^{N^{q}}}\sum_{n=1}^{N}(\sum_{i\in\mathcal{W}_{n}}\frac{1}{m}\log(1+e^{-y_{i}a_{i}^{T}x_{n}})+\frac{\tau}{N}\|x_{n}\|_{1})+\iota_{\mathcal{C}(x)},\eqno{(6.3)}$$
where $x = (x_{1}, ..., x_{N})$ is in  $\mathbb{R}^{N^{q}}$. It is
easy to see that problems (6.1), (6.2) and (6.3) are equivalent and
problem (6.3) is in the form of (4.2).

\noindent \textbf{Acknowledgements}

This work was supported by the National Natural Science Foundation
of China (11131006, 41390450, 91330204, 11401293), the National
Basic Research Program of China (2013CB 329404), the Natural Science
Foundations of Jiangxi Province (CA20110\\
7114, 20114BAB 201004).

\end{document}